\documentclass{article}
\usepackage{blindtext}
\usepackage[a4paper, total={6in, 8in}, includehead, headheight=5mm, margin=1.2in]{geometry}
\usepackage[utf8]{inputenc}
\usepackage{fancyhdr}
\usepackage{amsfonts}\usepackage{bbm}
\usepackage{amsthm}
\usepackage{amsmath}
\usepackage{enumerate}
\usepackage{mathtools}
\usepackage{tikz-cd}
\usepackage{thmtools}
\usepackage{graphicx}
\usepackage[ddmmyy]{datetime}
\usepackage{stmaryrd}
\usepackage{titlesec}
\usepackage{mathabx}
\usepackage{mathrsfs}
\usepackage{quiver}
\usepackage{enumitem}
\usepackage{scalerel}[2016/12/29]
\usepackage[T1]{fontenc}
\usepackage{tgtermes}
\usepackage[colorlinks,linkcolor=black]{hyperref}
\usepackage[nameinlink]{cleveref}

\newcommand{\HH}{\operatorname{HH}}

\newcommand{\bv}{\operatorname{BV}}

\newcommand{\Der}{\operatorname{Der}}
\let\phi\varphi 

\newcommand{\gphom}{\operatorname{Hom}_{\textrm{Group}}}

\theoremstyle{plain}
\newtheorem{theorem}{Theorem}
\newtheorem*{theorem*}{Theorem}

\newtheorem{lemma}[theorem]{Lemma}
\newtheorem{corollary}[theorem]{Corollary}
\newtheorem*{corollary*}{Corollary}
\newtheorem{example}[theorem]{Example}

\numberwithin{example}{section}
\numberwithin{theorem}{section}

\numberwithin{corollary}{section}

\theoremstyle{definition}

\newtheorem*{definition*}{Definition}
\newtheorem{remark}[theorem]{Remark}

\newtheorem{defn}[theorem]{Definition}
\makeatletter
\renewcommand{\env@cases}[1][@{}l@{\quad}l@{}]{%
  \let\@ifnextchar\new@ifnextchar
  \left\lbrace
  \def\arraystretch{1.2}%
  \array{#1}%
}
\makeatother

\newlist{pfparts}{enumerate}{1}
\setlist[pfparts,1]{%
  label=\arabic.,
  font=\normalfont\textbf,
  itemindent=2pt,
  wide,
  itemsep=0pt,topsep=2pt,
  labelsep=0.75ex
}
\makeatletter
\def\namedlabel#1#2{\begingroup
    #2%
    \def\@currentlabel{#2}%
    \phantomsection\label{#1}\endgroup 
}
\makeatother

\newtheorem{proposition}[theorem]{Proposition}

\declaretheoremstyle[
headfont=\color{black}\normalfont\bfseries,
notefont = \color{red}\normalfont,
bodyfont=\color{black}\normalfont\itshape,
]{coloredDefinition}

\declaretheoremstyle[
headfont=\color{black}\normalfont\bfseries,
notefont = \color{teal}\normalfont,
bodyfont=\color{black}\normalfont\itshape,
]{coloredExample}

\renewenvironment{abstract}
{\begin{quote}
\noindent\par{\scshape\abstractname.}}
{\medskip\noindent
\end{quote}
}

\title{On solvability of $\operatorname{HH}^1$ of odd $p$-groups.}
\author{\textsc{Matthew Antrobus}}
\date{}

\makeatletter
\let\inserttitle\@title
\makeatother

\titleformat{\section}
{\normalfont\large\scshape}{\thesection.}{1em}{}
\titleformat{\subsection}
  {\normalfont\normalsize\scshape}{\thesubsection.}{1em}{}

\begin{document}
\maketitle
\begin{abstract}
We give a necessary and sufficient criterion for the solvability of $\HH^1(kP)$ as a Lie algebra, where $P$ is a $p$-group with $p$ odd, in terms of a directed graph constructed from the group $P$. This gives non-trivial results on the structure of such Lie algebras. 
\end{abstract}
\section{Introduction}
The solvability of the first Hochschild cohomology of group algebras has been of great interest, and has been investigated in many papers (see \cite{BKL},  \cite{Eisele}, \cite{Wang}, \cite{Liu},  \cite{Degrassi}). Importantly, the Lie algebra structure of $\operatorname{HH}^1(kG)$ is a Morita invariant, and so has promise to resolve questions in Block theory. In particular, Hochschild cohomology, and specifically $\HH^1(kG)$, gives numerical invariants of Morita equivalence classes. In the present paper, we give a necessary and sufficient criterion for the Hochschild cohomology of the group algebra of a $p$-group to be solvable, for $p\not=2$. This criterion gives us other non-trivial results on their structure. \\
Let $p$ be a prime, and $k$ a field of characteristic $p$. Given $P$ a $p$-group, $\Phi(P)$ is the intersection of all maximal subgroups of $P$, and equivalently is the smallest subgroup so that $P/\Phi(P)$ is elementary abelian. Moreover, given an inclusion $H\rightarrow G$ of groups, there is a group homomorphism $\operatorname{tr}_H^G:G\rightarrow H/H'$. The definition of this map is nuanced and so we give it later. We will give a criterion for the solvability of $\operatorname{HH}^1(kP)$ in terms of a directed graph, defined as follows.
\begin{defn} Given a $p$-group $P$, we construct a directed graph $\Gamma(P)$ with:
\begin{itemize} 
\item vertices labelled by elements of $P/\Phi(P)$.
\item an edge from $a$ to  $b\in P/\Phi(P)$ when there are $x, y \in P$ with image $a,b$ in $P/\Phi(P)$ so that
$$\operatorname{tr}_{C_P(x)}^P(y) \not \in \Phi(C_P(x) / C_P(x)')$$
\end{itemize}
\end{defn}
\begin{theorem} When $p\not=2$, $\HH^1(kP)$ is solvable as a Lie algebra if and only if the directed graph $\Gamma(P)$ is acyclic. \end{theorem}

We call $\HH^1(kP)$ $n$-strongly solvable when $D^n(\HH^1(kP))$ is nilpotent. In particular, writing $l(\Gamma(P))$ for the length of a longest directed path in $\Gamma(P)$ we have the following result.
\begin{corollary} 
 When $p\not=2$, if $l(\Gamma(P))<n+1$ then $\HH^1(kP)$ is $n$-strongly solvable. If $l(\Gamma(P)) > n+1$ then $\HH^1(P)$ is not $n$-strongly solvable.
\end{corollary}
Given $p$-groups $P_1, P_2$, we can describe the Hochschild cohomology of their product $\nolinebreak\operatorname{HH}^1(k(P_1 \times P_2))$ in terms of $\HH^{j}(kP_i)$ where $j\in \{0,1\}$ and $i\in \{1, 2\}$. Crucially should $\HH^1(k(P_1 \times P_2))$ be solvable, then it is immediate that both of $\HH^1(kP_1)$ and $\HH^1(kP_2)$ are solvable. The converse is not true in general - taking $P_{1}, P_{2} = C_2$ gives a counterexample. However, our theorem gives us a converse for odd primes.
\begin{corollary} Take $p\not=2$, $p$-groups $P_1, P_2$, and write $P:=P_1 \times P_2$. $\HH^1(k P)$ is solvable if and only if $\HH^1(kP_1)$ and $\HH^1(kP_2)$ are solvable. \end{corollary} 
We further have partial results for $p=2$. We define a second directed graph as follows:
\begin{defn} $\Gamma_2(P)$ is a directed graph with 
\begin{itemize} 
\item vertices labelled by $P / \Phi(P)$.
\item an edge $a \rightsquigarrow b$, when there are edges $ a\rightarrow b$ and $a\rightarrow ab$ in $\Gamma(P)$.
\end{itemize}
\end{defn}
Which allows us to state the following result
\begin{corollary} If $\Gamma_2(P)$ has a cycle, then $\HH^1(kP)$ is not solvable.
\end{corollary}
\section{A reminder on the first Hochschild cohomology}
We briefly recall the properties of the first Hochschild cohomology. Given a $k$-algebra $A$, From the Bar Resolution, we obtain the isomorphism:
$$\HH^1(A)\cong \Der(A) / \operatorname{InnDer}(A).$$
Where $\Der(A)$ is the set of $k$-linear Derivations on $A$, which forms a Lie subalgebra of $\mathfrak{gl}(A)$, and $\operatorname{InnDer}(A)$ is the subalgebra of derivations of the form
$$a \mapsto [x,a] = xa-ax.$$
One can note that for $f\in \Der(A)$ and $x \in A$ that
$$[f,[x,-]] = [f(x),-],$$
so that inner derivations form an ideal, and $\HH^1(A)$ inherits the structure of a Lie Algebra. For more details the reader may consult \cite{Witherspoon}.
\subsection{Properties of solvable Lie Algebras}
Recall that a Lie algebra $\mathfrak{g}$ is called solvable when its derived series $D^n(\mathfrak{g})$ vanishes for sufficiently large $n$. This is defined as follows:
$$D^0(\mathfrak{g}) = \mathfrak{g}, D^{i+1}(\mathfrak{g}) = [D^i(\mathfrak{g}),D^i(\mathfrak{g})].$$
The minimal $n \in \mathbb{N}$ for which $D^n(\mathfrak{g})$ is zero is called the derived length of $\mathfrak{g}$, which we write $\operatorname{dl}(\mathfrak{g})$. We record the following important property of solvable Lie algebras which we will use repeatedly.
\begin{proposition}\label{solvablequotient} If $\mathfrak{h} \subset \mathfrak{g}$ is a Lie ideal of $\mathfrak{g}$, and both $\mathfrak{h}$ and $\mathfrak{g}/\mathfrak{h}$ are solvable, then $\mathfrak{g}$ is solvable.  \end{proposition}
\begin{proof} 
We note that as $q: \mathfrak{g} \rightarrow \mathfrak{g}/ \mathfrak{h}$ is a surjective Lie algebra homomorphism, $q([\mathfrak{g}, \mathfrak{g}]) = [\mathfrak{g}/\mathfrak{h}, \mathfrak{g}/\mathfrak{h}]$, and moreover 
$$q(D^i(\mathfrak{g}))= D^i(\mathfrak{g}/\mathfrak{h}).$$
So then, as $\mathfrak{g}/ \mathfrak{h}$ is solvable, we see that for some $i$, $D^i(\mathfrak{g}/\mathfrak{h})$ vanishes, which tells us that $D^i( \mathfrak{g})$ lands in $\mathfrak{h}$. Then $D^{i}(\mathfrak{g}) \subset D^0(\mathfrak{h})$, and by induction
$$D^{i+n} ( \mathfrak{g}) \subset D^n(\mathfrak{h}),$$
so as $\mathfrak{h}$ is solvable, this also vanishes.
\end{proof}
\begin{corollary}\label{derlength}
If $\mathfrak{h}$ is a solvable ideal of the solvable Lie algebra $\mathfrak{g}$ then
$$\operatorname{dl}(\mathfrak{g}) \leq \operatorname{dl}(\mathfrak{h}) + \operatorname{dl}(\mathfrak{g}/\mathfrak{h}).$$
\end{corollary}
A Lie algebra $\mathfrak{h}$ is called nilpotent when its lower central series $C^n(\mathfrak{h})$ vanishes for large $n$. This is defined as follows:
$$C^0(\mathfrak{h}) = \mathfrak{h}, C^n(\mathfrak{h}) = [\mathfrak{h}, C^{n-1}(\mathfrak{h})].$$
Should a Lie algebra $\mathfrak{g}$ be solvable, it is called strongly solvable when its derived algebra $D^1(\mathfrak{g})$ is nilpotent. More generally, we will say that $\mathfrak{g}$ is $n$-strongly solvable when $D^n(\mathfrak{g})$ is nilpotent. We will need the following results on $n$-strongly solvable Lie algebras.
\begin{proposition}\label{nstrongsolv}
Given a Lie algebra $\mathfrak{g}$ with a nilpotent ideal $\mathfrak{h}$, we have the following.
\begin{enumerate} 
\item If $\mathfrak{h}$ is nilpotent, and $D^n(\mathfrak{g}/\mathfrak{h})$ is zero then $\mathfrak{g}$ is $n$-strongly solvable.
\item If $\mathfrak{g}$ is $n$-strongly solvable then $\mathfrak{g}/\mathfrak{h}$ is $n$-strongly solvable
\end{enumerate}
\end{proposition}
\begin{proof}[Proof of 1.] As $D^n(\mathfrak{g}/\mathfrak{h})$ is zero, we see that $D^n(\mathfrak{g}) \subset \mathfrak{h}$. But $\mathfrak{h}$ is nilpotent, and subalgebras of nilpotent algebras are nilpotent.
\end{proof}
\begin{proof}[Proof of 2.] We noted before that, writing $q: \mathfrak{g} \rightarrow \mathfrak{g}/\mathfrak{h}$, we have that $q(D^n(\mathfrak{g})) = D^n(\mathfrak{g}/\mathfrak{h}, \mathfrak{g}/\mathfrak{h})$. Moreover, we know that homomorphic images of nilpotent algebras are nilpotent, so that as $D^n(\mathfrak{g})$ is nilpotent, so is $D^n(\mathfrak{g}/\mathfrak{h}, \mathfrak{g}/\mathfrak{h})$.
\end{proof}
\subsection{The Centralizer Decomposition of Hochschild cohomology}
Given a group $G$, with a complete set of conjugacy class representatives $x_1,\ldots, x_n$ there is a well known decomposition of the Hochschild cohomology of $\operatorname{HH}^*(kG)$ given as follows:
$$\HH^*(kG) = \bigoplus_{j=1,\ldots,n} \operatorname{H}^*(C_G(x_j), k),$$
in terms of the group cohomologies of centralizers of our conjugacy classes. In \cite{Liu} the isomorphism is made explicit. Recall that any cocycle in $\operatorname{H}^n(C_G(x_j),k)$ may be represented by a function $f:C_G(x_j)^n\rightarrow k$, emerging from the Bar Resolution. In the same vein, we can represent a cocycle in $\HH^n(kG)$ by a function $g:G^n \rightarrow kG$. Given such a function $f:C_G(x_j)^n \rightarrow k$, and a set $\gamma_1,\ldots, \gamma_m$ of right coset representatives of $C_G(x_j)$, where we assume that $\gamma_1 = e$, the corresponding coycle in Hochschild cohomology is given by:
$$\hat{f}(g_1,\ldots,g_n) = (-1)^{\frac{n(n+1)}{2}}g_1\ldots g_n \sum_{i=1}^m f(h'_{i,1}, \ldots, h'_{i,n}) \gamma_i^{-1}x_j\gamma_i,$$
where the $h'_{i,1}, \ldots h'_{i,n}$ are determined to be the unique elements of $C_G(x_j)$ satisfying 
$$\gamma_i g_n^{-1} = h'_{i,1}\gamma_{s_i^1},\; \gamma_{s_i^1}g_{n-1}^{-1}=h'_{i,2}\gamma_{s_i^2}, \;\ldots, \;\gamma_{s_i^{n-1}}g_1^{-1} = h'_{i,n} \gamma_{s^n_i}.$$
For our usage, we will only consider $G$ abelian, and $n=1$. This simplifies things a great deal. In this case, $G$ is the complete set of conjugacy class representatives, and $H^1(C_G(x), k) = H^1(G, k) = \operatorname{Hom}_{\textrm{Group}}(G, k)$. So then for each $x$, $C_G(x) = G$ only has one coset so $m=1$. Moreover, given $g\in G$, the corresponding value $h'_{1,1}=g^{-1}$. So our formula reduces to 
$$\hat{f}(g) = -gf(g^{-1})x = f(g)gx.$$
So then, the first Hochschild cohomology of $kA$ for $A$ an abelian group is
$$\HH^1(kA) = \bigoplus_{a\in A} \gphom(A,k).$$
For $\varphi \in \gphom(A,k)$ and $a\in A$, we write $(\varphi, a)$ for the element $\varphi$ in the direct summand labelled by $a$. From the above then,
$$(\phi,a)(z) = \phi(z)za,$$
 and from here we can compute commutators.
$$\begin{aligned} [(\phi,a),(\psi,b)](z) &= (\phi,a)(\psi,b)(z) - (\psi,b)(\phi,a)(z) \\
&=(\phi,a)(\psi(z)bz) - (\psi,b)(\phi(z)az)\\
&=\phi(bz)\psi(z)abz - \psi(az)\phi(z)abz \\
&=\phi(b)\psi(z)abz - \psi(a)\phi(z)abz + \phi(z)\psi(z)abz - \psi(z)\phi(z)abz \\
&= (\phi(b)\psi - \psi(a)\phi)(z)abz.
\end{aligned}$$
So, we have proven:
\begin{proposition} Given $\phi, \psi \in \gphom(A,k)$ and $a,b \in A$ we have
$$[(\phi,a), (\psi,b)] = (\phi(b)\psi - \psi(a)\phi, ab).$$
\end{proposition}
We record one more thing before proceeding. As $kG$ is a Symmetric Frobenius Algebra, $\HH^*(kG)$ comes equipped with an operator of degree $-1$ called the $\bv$-operator, denoted $\Delta$. In particular, given a derivation $f: kG \rightarrow kG$, $\Delta(f)$ is the unique element obeying the following identity
$$\langle f(x), 1 \rangle = \langle \Delta(f), x \rangle.$$
So, specialising again to the abelian case, given $\phi\in\gphom(A,k)$ and $a \in A$ we see that
$$\langle (\phi,a)(x), 1 \rangle = \langle \phi(x)ax, 1 \rangle.$$
This inner product vanishes apart from when $x=a^{-1}$, and then is equal to
$$\langle \phi(a^{-1}) , 1 \rangle = -\phi(a),$$
so that $\Delta((\phi,a)) = -\phi(a)a$. 
\section{The Lie Algebra $\HH^1(kP)$}
\subsection{The Jacobson Filtration}
Recall that the Jacobson Radical $J$ of $kG$ is the largest nilpotent right ideal of $kG$. In the case where $G$ is a $p$-group (which will be distinguished as the group $P$), this is the augmentation ideal, which is generated by elements of the form $g-1$ for $g \in P$. We record some facts from \cite{Wang}. Write
$$\Der_m(kG) := \{ f \in \Der(kG): f(kG) \subset J^m \}.$$
Crucially from Proposition 2.3 of \cite{Wang}, we know that $\Der_2(kG)$ is a nilpotent subalgebra. We also record Lemma 2.5 from their paper. Writing $G' = [G,G]$, we have a ring homomorphism
$$\begin{aligned}q : kG &\rightarrow kG/G' \\
g &\mapsto gG', \end{aligned}$$
whose kernel is spanned by elements of the form $g-1$, where $g\in G'$. It is noted in the Lemma, the kernel of this map is preserved by all derivations, so that it induces a well defined Lie Algebra homomorphism $\HH^1(kG) \rightarrow \HH^1(kG/G').$ First, we have an initial observation.
\begin{proposition} If $\HH^1(kG)$ is solvable, so is its image in $\HH^1(kG/G')$.\end{proposition}
This is a consequence of solvability descending to quotients. As the kernel of this map need not necessarily be solvable for $G$ a general group, we cannot state a converse. However, this will allow us to demonstrate insolvability for general groups. From here on, our results will be regarding $p$-groups. 
\begin{proposition} For $P$ a p-group and $q:kP \rightarrow kP/P'$ as above, we have $\operatorname{Ker}(q) \subset J^2$. \end{proposition}
\begin{proof}
We note that $\operatorname{Ker}(q)$ is a $2$-sided ideal, generated by elements of the form $xyx^{-1}y^{-1}-1$. $J$ is the augmentation ideal, and so a $2$-sided ideal, and therefore so is $J^2$. We see that
$$xyx^{-1}y^{-1}-1 = ((x-1)(y-1) - (y-1)(x-1))x^{-1}y^{-1},$$
which is clearly in $J^2$.
 \end{proof}
\begin{proposition} $\operatorname{InnDer}(kP) \subset \Der_2(kP)$. \end{proposition}
\begin{proof} We know each inner derivation is of the form $[x,-]$ for some $x\in kP$ and that
$$[x,a] = xa-ax \in \operatorname{Ker}(q) \subset J^2,$$
so we are done.
\end{proof}
\begin{corollary} $\HH^1(kP)$ is solvable if and only $\operatorname{Der}(kP)$ is.\end{corollary}
\begin{proof} We know that $\operatorname{InnDer}(kP)$ is contained in $\Der_2(kP)$ a nilpotent subalgebra. So by Proposition \ref{solvablequotient} we have our result.
\end{proof}
\begin{proposition} $\HH^1(kP)$ is solvable if and only if its image in $\HH^1(kP/P')$ is solvable.
\end{proposition}
\begin{proof}
Given $f \in \HH^1(kP)$, the morphism $\mathfrak{q}: \HH^1(kP) \rightarrow \HH^1(kP/P')$ is given as follows:
$$\mathfrak{q}(f)(q(a)) = qf(a),$$
which is well defined as the subspace $\operatorname{Ker}(q)$ is preserved by all derivations. So then we have that a morphism $f$ is in  $\operatorname{Ker}(\mathfrak{q})$ when $f(kG) \subset \operatorname{Ker}(q)$. But we know $\operatorname{Ker}(q) \subset J^2$, so that $\operatorname{Ker}(\mathfrak{q})\subset \operatorname{Der}_2(kP) / \operatorname{InnDer}(kP)$. So as the kernel of $\mathfrak{q}$ is contained in a nilpotent subalgebra, it is nilpotent, and we obtain our result by using Proposition \ref{solvablequotient}.
\end{proof}
So then, we only need to understand the image of our map $\mathfrak{q}$. 

\subsection{Reduction to the Elementary Abelian case}
We have shown that solvability of $\HH^1(kP)$ is determined by the solvability of a subalgebra of its abelianisation, $\HH^1(kP/P')$. So that we will not have to worry about the order of elements in future, it turns out we can reduce to an elementary abelian group.
\begin{defn} Given a group $G$, its $p$-Frattini subgroup, denoted $\Phi_p(G)$ is the minimal normal subgroup for which $G/\Phi_p(G)$ is an elementary abelian $p$-group. \end{defn}
Note that this group is well defined, as given subgroups $A,B$ for which $G/A, G/B$ are elementary abelian $p$-groups, we have an injection
$$G/(A\cap B) \rightarrow G/A \times G/B,$$
ensuring that $G / (A \cap B)$ is an elementary abelian $p$-group. 
\begin{proposition} $\Phi_p(G)$ is the normal subgroup generated by $G^p$ and $[G,G]$.\end{proposition}
\begin{proof} We note that as $G/\Phi_P(G)$ is an elementary abelian $p$-group, we must have $G^p, [G,G] \subset \Phi_p(G)$. So it only remains to show that the quotient of $G$ by this subgroup is an elementary abelian $p$-group. But it is surely abelian, as we have quotiented by $[G,G]$, and every element has order dividing $p$, as we have quotiented by the subgroup generated by $G^p$. \end{proof}
For more details, the reader can consult \cite{Gruenberg}.
\begin{proposition} The map $kG \rightarrow kG/\Phi_p(G)$ induces a Lie algebra map $\HH^1(kG)\rightarrow \HH^1(kG/\Phi_p(G))$, and the solvability of $\HH^1(kG)$ implies the solvability of its image in $\HH^1(kG/\Phi_p(G))$. Moreover, if $G$ is a $p$-group, which we write as $P$, $\HH^1(kP)$ is solvable if and only if its image in $\HH^1(kP/\Phi(P))$ is solvable. \end{proposition}
\begin{proof} 
Suppose we have $z\in Z(G)$ so that $z^p\not=0$. Then for any derivation $f:kG \rightarrow kG$, we have
$$f(z^p) = pz^{p-1}f(z)=0,$$
so that all derivations vanish on $Z(G)^p$. Write the ring homomorphism $r: kG \rightarrow kG/\langle z^p \rangle$. The kernel of this map is spanned by elements of the form $(z^{np}-1)a$, for $a\in G$. We see that
$$f((z^{np}-1)a)= (z^{np}-1)f(a) \in \operatorname{Ker}(r),$$
so that all derivations preserve the kernel of $r$. Therefore we get a well defined Lie Algebra homomorphism $\mathfrak{r}: \operatorname{HH}^1(kG) \rightarrow \operatorname{HH}^1(k(G/\langle z^p \rangle))$. The kernel of this map consists of those derivations whose image lands in the kernel of $r$. We note that in the situation where $G$ is a $p$-group $P$,
$$(z^{np}-1) = (z^n-1)^p \in J^2,$$
so that $\operatorname{Ker}(\mathfrak{r}) \subset \Der_2(kP)$. So then $\HH^1(kP)$ is solvable if and only if its image in $\HH^1(k(P/\langle z^p \rangle))$ is solvable. Note that from our previous proposition, writing $q:G\rightarrow G/G'$, we have $\Phi_p(G) = q^{-1}((G/G')^p)$. So we get morphisms
$$\HH^1(kG) \rightarrow \HH^1(kG/G') \rightarrow \HH^1(k((G/G')/(G/G')^p)) \cong \HH^1(kG/\Phi_p(G)).$$
Then solvability of $\HH^1(kG)$ implies solvability of the image of this compositon. When $G$ is a $p$-group, the kernel of both of these maps are solvable, which implies solvability of the kernel of the composition. So then, $\HH^1(kP)$ is solvable if and only its in image in $\HH^1(kP/\Phi(P))$ is solvable.
\end{proof}

\subsection{Transfer for Groups and Morphisms}
Recall that given groups $H \subset G$, we have a morphism $\operatorname{tr}_{H}^G: G \rightarrow H/H'$, described as follows. Take a right transversal $t_1,\ldots, t_m$ of $H$ in $G$, and for each $g \in G$ define
$$\operatorname{tr}_{H}^G(g) = (\prod_{i=1}^m g_{[i]})H',$$
where $g_{[i]}$ is defined so that $t_i g = g_{[i]} t_j$, for some $g_{[i]} \in H$. For more details the reader can see \cite{Huppert}.
From our previous discussion, we know that given $\phi:C_G(x) \rightarrow k$, we obtain a derivation via
$$\hat{\phi}(g) = -g\sum_{i=1}^m \phi(g_{i}) \gamma_i^{-1}x\gamma_i,$$
where $g_i \in C_G(x)$ is so that $\gamma_i g^{-1} = g_i \gamma_{s_i}$. We see that applying $q$ we will get the image of our derivation in $\HH^1(kG/\Phi_p(G))$. This will be of the form
$$\mathfrak{q}(\hat{\phi})(q(g)) = q(g)q(x)\phi(\prod_{i=1}^m g_i),$$
where notably, $\prod_{i=1}^m g_i$ is the transfer map $\operatorname{tr}_{C_G(x)}^G(g^{-1})$. From the formula above, it can be seen that given $\phi:C_G(x)\rightarrow k$, which we regard as an element of $\HH^1(kG)$ via the centralizer decomposition, we have
$$\hat\phi \xmapsto{\mathfrak{q}} (-\phi\cdot \operatorname{tr}_{C_G(x)}^G, q(x)),$$
writing this categorically, the map
$$\HH^1(kG) \rightarrow \HH^1(kG/\Phi_p(G))$$
descends to direct summands as
$$\operatorname{H}^1(C_G(x), k) \xrightarrow{(-\operatorname{tr}_{C_G(x)}^G)^*} \operatorname{H}^1(G/\Phi_p(G), k) \cong \operatorname{H}^1(G, k).$$
Notably, the Lie algebra structure only depends on the image of this map, which is independent of the minus sign in our precomposition. Further, also note we have slightly abused notation here; $\operatorname{tr}_{C_G(x)}^G$ does not land in $C_G(x)$, but rather $C_G(x)/C_G(x)'$. As we have the isomorphisms
$$\operatorname{Hom}(C_G(x), k) \cong \operatorname{Hom}(C_G(x)/C_G(x)', k)\cong \operatorname{Hom}(C_G(x)/\Phi_p(C_G(x)), k)$$
this abuse is justified. From here on, we will slightly abuse notation by associating maps $\phi: G \rightarrow k$ with their factor ${\phi}: G / \Phi_p(G) \rightarrow k$.
\section{The image of $\mathfrak{q}$}
To proceed, we introduce some notation. Write $A := G/ \Phi_p(G)$,  $\mathfrak{g} := \HH^1(kA)$ and $\mathfrak{h} = \operatorname{Im}(\mathfrak{q})$. Write $\mathfrak{g}_a$ for the $a$-labelled summand of $\HH^1(kA)$ and $\mathfrak{h}_a := \mathfrak{g}_a \cap \mathfrak{h}$. We first note that $h_e := g_e$, as transfer from $G$ to $G$ is the identity. \\ 

We will give our criteria for solvability in terms of a directed graph. We create a directed graph with edges labelled by the elements of $A$. Then, we draw an edge from $a$ to $b$ when there exist elements $x, y \in G$ so that $q(x)=a$ and $q(y)=b$, and a group homomorphism $\alpha:C_G(x)\rightarrow k$, with $\nolinebreak \alpha \circ \operatorname{tr}_{C_G(x)}^G(y) \not=0$. Call this directed graph $\Gamma(G)$.

Recall that as $\gphom(C_G(x), k) \cong \gphom(C_G(x)/\Phi_p(C_G(x)), k)$, we see that such an $\alpha:C_G(x) \rightarrow k$ exists if and only if $\operatorname{tr}_{C_G(x)}^G(y) \not \in \Phi_p(C_G(x)/C_G(x)')$.
We will proceed to prove the following result.
\begin{proposition} If $p\not=2$, $\mathfrak{h}$ is solvable if and only if $\Gamma(G)$ is acyclic. \end{proposition}
\begin{remark} When discussing cycles and paths in a directed graph $\Sigma$, we will always mean directed paths and directed cycles. That is, paths where edges are only able to be traversed from their start to their end, and not backwards. In this vein, that $\Sigma$ is acyclic means it contains no directed cycles. \end{remark}
This has the following consequence.
\begin{theorem} If $p \not=2$, $\HH^1(kP)$ is solvable if and only if $\Gamma(P)$ is acyclic. \end{theorem}
The proof of this is performed in the following sections. Before that, we define a reduced version of the graph which will be used later for precise statements about derived length.
\begin{defn} The reduced graph $\bar{\Gamma}(G)$ is a subgraph of $\Gamma(G)$ with the same vertices. Suppose there is an edge from $a$ to $b$ in $\Gamma(G)$. This edge is in $\bar{\Gamma}(G)$ exactly when $\mathfrak{h}_b\not=0$. \end{defn}

We breifly explain why the reduction is useful. We want that whenever there is an edge from $a$ to $b$ in $\bar\Gamma(G)$, but no loops at $a$, that $[\mathfrak{h}_a, \mathfrak{h}_b]\not=0$. Suppose we are in the scenario we just described. Then I must have $x,y \in G$ with $q(x)=a, q(y)=b$, and $\phi := \alpha \circ \operatorname{tr}_{C_G(x)}^G$ so that $\phi(y)\not=0$. Moreover, as there are no loops at $a$, I know that $\phi(x)=0$. If $\mathfrak{h}_b \not=0$, then there is some $y'$ with $q(y')=b$, and $\beta:{C_G(y')}\rightarrow G$ so that $\psi := \beta \circ \operatorname{tr}_{C_G(y')}^G \not =0$. Crucially, $ \phi$ and $\psi$ are maps $G \rightarrow k$ and so descend to maps $G/ \Phi_p(G) \rightarrow k$. Therefore, we have $\phi(y') = \phi(y) \not =0 $. Plugging this information into our formula for the Lie bracket, we see
$$[(\phi, a), (\psi, b)] = (\phi(b)\psi, ab) \not=0 $$
giving the result as needed. Moreover, the same formula holds for any $(\psi,b) \in \mathfrak{h}_b$ so that $[\mathfrak{h}_a, \mathfrak{h}_b] = \mathfrak{h}_ba$. For instance it is easy to see the following:
\begin{proposition} $\mathfrak{h}$ is abelian if and only if $\bar{\Gamma}(G)$ has no edges. \end{proposition}
But further, $\mathfrak{h}_b$ is non-zero exactly when $b$ has some outgoing edge. So then, we have the following.
\begin{proposition}\label{terminaledge} $\bar{\Gamma}(G)$ is obtained from $\Gamma(G)$ by removing terminal edges. 
\end{proposition}
\subsection{Reduction to the Kernel of $\Delta$}
 We proceed with a series of reductions. We first deal with the case where $\Gamma(G)$ contains a loop.
 \begin{proposition} Suppose $p\not=2$ and there is $a\in G$  and $\alpha:C_G(a) \rightarrow k$, where writing $\phi := -\alpha \circ \operatorname{tr}_{C_G(a)}^G$ we have $\phi(a)\not=0$. Then $\mathfrak{h}$ is not solvable. 
\end{proposition}
\begin{proof} 
So we know that $(\phi, a) \in \mathfrak{h}$, but then also $(\phi, a^{-1}) \in \mathfrak{h}$ as transfer relies only on the subgroup $C_G(z) = C_G(z^{-1})$. Further, we have $(\phi, e) \in \mathfrak{h}$ as $\mathfrak{h}_e =\mathfrak{g}_e$. Now
$$[(\phi, a^{\pm 1}), (\phi, e)] = (\pm\phi(a)\phi,a^{\pm 1}),$$
and
$$\begin{aligned}[(\phi,a), (\phi,a^{-1})] &= (\phi(a^{-1})\phi - \phi(a)\phi, e)\\
&=(-2\phi(a)\phi, e) .\end{aligned}$$
So then as $\phi(a)\not=0$, we see that the subspace $\operatorname{Span}_k((\phi, z), (\phi, a^{-1}), (\phi, e))$ is a perfect subalgebra. All subalgebras of a solvable algebras are solvable, and perfect algebras are not solvable. So $\mathfrak{h}$ may not be solvable.
\end{proof}
This proof is where our restriction that $p \not = 2$ arises from. We note that this is not true when $p=2$. Considering at $P := C_2$, and $\HH^1(\mathbb{F}_2P)$ we have a loop, but the Lie algebra is solvable. \\

So then, without restriction  of generality in the case where $p\not=2$, we can assume that $\phi(a)=0$ for all $\phi, a$ as above. Incidentally, this is exactly that $\mathfrak{h} \subset \operatorname{Ker}(\Delta)$.
\subsection{Reduction to the case ($\dagger$)}
We now consider the scenario where have a cycle of two edges. That is, we have $a, b \in G/ \Phi_p(G)$, an edge from $a$ to $b$, and from $b$ to $a$ in $\Gamma(G)$. Once we have eliminated these two cases, the structure of our Lie algebra becomes significantly simpler.
\begin{proposition} Suppose there are $a,b\in G/ \Phi_p(G)$, with preimages $x,y \in G$, and group homomorphisms $\alpha:C_G(x) \rightarrow k$, $\beta: C_G(y) \rightarrow k$, where writing $\phi := \alpha \circ \operatorname{tr}_{C_G(x)}^G$ and $\psi := \beta \circ \operatorname{tr}_{C_G(y)}^G$, we have
$$\phi(y) \not=0, \; \; \psi(x)\not=0,$$
then $\mathfrak{h}$ is not solvable.   \end{proposition}
\begin{proof} 
If $a=b$ then we are in the situation of the last proposition, so we may assume that $a\not= b$ and moreover that no such $a$ exists. Thus $\phi(a) = \psi(b)=0$. We investigate the Lie subalgebra $\mathfrak{a}$ of $\mathfrak{h}$ generated by $(\phi, a)$, $(\psi,b)$. Then we claim that 
$$\operatorname{Ad}((\phi,a))^n((\psi,b)) = \phi(b)^{n-1}(\phi(b)\psi - n\psi(a)\phi, a^nb),$$
where recall we write $\operatorname{Ad}(f) = [f,-]$ for $f \in \mathfrak{h}$. We proceed by induction. For the base case $n=1$, we found earlier that
$$[(\phi, a), (\psi,b)] = (\phi(b)\psi - \psi(a)\phi, ab),$$
just as claimed. Now, suppose the claim is true for $n$. Then
$$\begin{aligned}\operatorname{Ad}^{n+1}((\phi,a))((\psi,b)) &= \operatorname{Ad}((\phi,a))(\operatorname{Ad}^{n}((\phi,a))((\psi,b))) \\
&= \operatorname{Ad}((\phi,a))(\phi(b)^{n-1}(\phi(b)\psi - n\psi(a)\phi, a^nb))\\
&= \phi(b)^{n-1}[(\phi, a), (\phi(b)\psi - n\psi(a)\phi, a^nb)] \\
&= \phi(b)^{n-1}(\phi(a^nb)(\phi(b)\psi-n\psi(a)\phi) - (\phi(b)\psi(a)-n\psi(a)\phi(a))\phi, a^{n+1}b) \\
&= \phi(b)^{n-1}(\phi(b)^2\psi - n\phi(b)\psi(a)\phi - \phi(b)\psi(a)\phi, a^{n+1}b) \\
&= \phi(b)^n(\phi(b)\psi - (n+1)\psi(a)\phi, a^{n+1}b),
\end{aligned}$$
as required. So notably,
$$\operatorname{Ad}((\phi,a))^p((\psi,b)) =(\phi(b)^p\psi, b),$$
as all elements in $A$ have order $p$. But then clearly $(\psi,b) \in [\mathfrak{a}, \mathfrak{a}]$. Likewise, we see $(\phi, a)\in [\mathfrak{a}, \mathfrak{a}]$. But as $\mathfrak{a}$ was the subalgebra generated by $(\psi,b)$ and $(\phi,a)$, we have $\mathfrak{a}\subset [\mathfrak{a}, \mathfrak{a}] $. So $\mathfrak{a}$ is perfect, and $\mathfrak{h}$ is not solvable.
\end{proof}
So, the only remaining case is that for each $a, b \in G/ \Phi_p(G)$, we may have an edge from $a$ to $b$, or an edge $b$ to $a$, but not both. In terms of elements, with $\alpha, \beta, \phi, \psi$ as above, at most one of $\phi(b), \psi(a) $ is non zero. We describe this as property $(\dagger)$. 
\subsection{The property $(\dagger)$}
For now, we write $a \rightarrow b$ to indicate an edge from $a$ to $b$ in our graph $\Gamma(G)$. 
\begin{proposition} Suppose that $\Gamma(G)$ contains a cycle. Then $\mathfrak{h}$ is not solvable. \end{proposition}
\begin{proof}
Suppose our graph contains a cycle, which we may assume is minimal, and of length greater than 2, so that we have $a_1\rightarrow a_2 \rightarrow \ldots\rightarrow a_n \rightarrow a_1$. This implies we have morphisms $(\phi_i, a_i)$ with the property that $\phi_i(a_{i+1})\not=0$, and that $\phi_n(a_1)\not=0$. Because of property $(\dagger)$, we know that $\phi_{i+1}(a_i)=0$, and that $\phi_i(a_i)=0$. As before, we consider the subalgebra $\mathfrak{b}$ generated by $\{(\phi_i, a_i): i=1,\ldots, n \}$. Notably,
$$[(\phi_i,a_i), (\phi_{i+1},a_{i+1})] = (\phi_i(a_{i+1})\phi_{i+1}, a_ia_{i+1}),$$
and so,
$$\operatorname{Ad}((\phi_i,a_i))^p(\phi_{i+1}, a_{i+1}) = \phi_i(a_{i+1})^p(\phi_{i+1}, a_{i+1})$$
So then clearly $(\phi_{i},a_i) \in [\mathfrak{b},\mathfrak{b}]$, for $1<i\leq n$. In exactly the same way,
$$\operatorname{Ad}(\phi_n,a_n)^p(\phi_1,a_1) = \phi_n(a_1)^p(\phi_1, a_1),$$
so $\mathfrak{b}$ is perfect, and we may not have that $\mathfrak{h}$ is solvable.
\end{proof}

\begin{proposition}\label{mainprop} Suppose $\Gamma(G)$ is acyclic. Then $\mathfrak{h}$ is solvable.\end{proposition}
\begin{proof} Suppose $\Gamma(G)$ is acyclic. For $S \subset A$ write:
$$\mathfrak{h}_S := \bigoplus_{s \in S} \mathfrak{h}_s,$$
and suppose we have a subalgebra $\mathfrak{a} \subset \mathfrak{h}_S$, satisfying:
$$\mathfrak{a} = \bigoplus_{a \in A} \mathfrak{a} \cap \mathfrak{g}_a.$$
As $\Gamma(G)$ is acyclic, $S$ must contain a minimal element, by which we mean a vertex with no incoming edges. Let $m$ be such a minimal element. I seek to show that
$$[\mathfrak{a}, \mathfrak{a}] \subset \mathfrak{h}_{S \backslash \{m\}}$$
We decompose $\mathfrak{h}_S := \mathfrak{h}_{S \backslash \{m\}} \oplus \mathfrak{h}_m$, and consider the 3 cases of brackets. Because of our assumption that $\mathfrak{a}$ is the direct sum of it's graded parts, this will suffice. First note that by our previous work, we know that $\mathfrak{h}_m$ is abelian (as $\mathfrak{h} \subset \operatorname{Ker}(\Delta)$). So, take $(\phi,m) \in \mathfrak{h}_m \cap \mathfrak{a}$ and $(\psi, b) \in \mathfrak{h}_b  \cap \mathfrak{a}$, with $b \in S \backslash \{m \}$. Then as $m$ is minimal, we know that $\psi(m)=0$. So we have 
$$[(\phi,m), (\psi,b)] = (\phi(b)\psi, mb),$$
and if $\phi(b)=0$ we're done. If not then as $\mathfrak{a}$ is a subalgebra, this still lies in $\mathfrak{h}_S$, but not in $\mathfrak{h}_m$, as $b\not=e$ (or else $\phi(b)=0$).  \\
So then we reduce to the case where $(\phi, a), (\psi,b) \in \mathfrak{a}$, with $a,b \not= m$. Suppose that $ab=m$. Then we know that $\phi(a)=0$, and also that $\phi(m)=0$. So then $\phi(b)=0$. Similarly, $\psi(b)=0$ and $\psi(m)=0$, so $\psi(a)=0$. So then, 
$$[(\phi,a), (\psi,b)] =0,$$
and $[\mathfrak{a}, \mathfrak{a}] \subset \mathfrak{h}_{S \backslash \{ m \}}$, as claimed. \\
From here, note that if $\mathfrak{a}$ satisfies
$$\mathfrak{a} = \bigoplus_{a\in A} \mathfrak{a} \cap \mathfrak{g}_a,$$
that then $[\mathfrak{a}, \mathfrak{a}]$ also satisfies this. We know of course that $$\bigoplus_{a \in A}[\mathfrak{a}, \mathfrak{a}] \cap \mathfrak{g}_a \subset [\mathfrak{a} , \mathfrak{a}],$$
but we see
$$\begin{aligned} [\mathfrak{a}, \mathfrak{a}] &= \bigoplus_{a \in A} \bigoplus_{b\in A} [\mathfrak{a}\cap \mathfrak{g_a}, \mathfrak{a} \cap \mathfrak{g_b}] \\
&\subset \bigoplus_{a \in A} \bigoplus_{b\in A} [\mathfrak{a},\mathfrak{a}] \cap [\mathfrak{g}_{a}, \mathfrak{g}_b] \\
&\subset  \bigoplus_{a \in A} \bigoplus_{b\in A} [\mathfrak{a},\mathfrak{a}] \cap \mathfrak{g}_{ab} \subset \bigoplus_{a \in A}[\mathfrak{a}, \mathfrak{a}] \cap \mathfrak{g}_a\end{aligned}  . $$
So that we have equality.\\
Choose an ordering of $A$ so that $a_i$ is a minimal element amongst $A \backslash \{a_1, \ldots, a_{i-1} \}$. Therefore, using the previous part of this proof we see that $D^i([\mathfrak{h}, \mathfrak{h}]) \subset \mathfrak{h}_{A \backslash \{ a_1, \ldots, a_i \}}$. This of course vanishes for large $i$, so that $\mathfrak{h}$ is solvable.
\end{proof}
Given a directed graph $\Sigma$, write $l(\Sigma)$ for the longest directed path in $\Sigma$. From Proposition \ref{terminaledge}, we know that $l(\Gamma(G)) = l(\bar\Gamma(G))+1$. 
\begin{proposition} The derived length of $\mathfrak{h}$ is given by $l(\Gamma(G))$. \end{proposition}
\begin{proof} We prove that the derived length is $l(\bar\Gamma(G))+1$, which gives the claim. First note that if we have a cycle in $\Gamma(G)$, we must also have a cycle in $\bar\Gamma(G)$. If we have an edge coming out of $a \in G / \Phi_p(G)$, then this posits the existence of a non-zero $(\psi,a) \in \mathfrak{h}_a$, so that $\mathfrak{h}_a\not=0$. In any cycle every vertex must have an outgoing edge, and so any incoming edge in $\Gamma(G)$ will also be in $\bar\Gamma(G)$. If $\mathfrak{h}$ is not solvable, it has a cycle in $\Gamma(G)$ and so in $\bar\Gamma(G)$, so that the infinite length of a path and the infinite derived length agree. From here, suppose that $\Gamma(G)$ is acyclic. Define the following subsets
$$T_n := \{b \in G /\Phi_p(G): \textrm{ There exists a path in } \bar\Gamma(G) \textrm{ of length greater or equal to } n \textrm{ ending at } b \}.$$
We seek to show that $D^n(\mathfrak{h}):= \mathfrak{h}_{T_n}$. This is certainly true for $n=0$. So then we proceed by induction. Suppose the maximal length of path ending at $b$ is of length $n$. Then, $b$ is a minimal element in $T_n$, so that by the proof of our last result we know that 
$$D^{n+1}(\mathfrak{h}) \subset \mathfrak{h}_{T_{n+1}}.$$
Suppose that $b$ has a path of length at least $n+1$ terminating at $b$. Then there is $a\in T_n$ so that there is an edge from $a$ to $b$ in $\bar\Gamma(G)$. This tells us that there is $(\phi, a) \in \mathfrak{h}_a \subset D^n(\mathfrak{h})$ so that $\phi(b)\not=0$. Moreover, as $\Gamma(G)$ is acyclic we know that $\phi(a)=0$. Thus
$$[(\phi,a), (\psi, b)] = (\phi(b)\psi, ab),$$
and as before,
$$\operatorname{Ad}((\phi,a))^p((\psi,b)) = \phi(b)^p(\psi, b),$$
so that $(\psi,b) \in D^{n+1}(\mathfrak{h})$. As $\psi$ was arbitrary, we see that $\mathfrak{h}_b \subset D^{n+1}(\mathfrak{h})$. So $D^{n+1}(\mathfrak{h}) = \mathfrak{h}_{T_{n+1}}$. Thus the derived length is one greater than the length of a maximal path in $\bar\Gamma(G)$.
\end{proof}
From here on, our results will concern $p$-groups. Write $\ell\ell(kP)$ for the minimal $n>0$ for which $J^n =0$. This is known as the Loewy length of $kP$.
\begin{corollary} The derived length of $\HH^1(kP)$ is bounded above by 
$$\operatorname{log}_2(\ell\ell(kP)-1) + l(\Gamma(P)).$$
\end{corollary}
\begin{proof} We know from Corollary \ref{derlength} that 
$$\operatorname{dl}(\HH^1(kP)) < \operatorname{dl}(\operatorname{Der}_2(kP)) + \operatorname{dl}(\mathfrak{h}).$$
Further, from $\cite{Wang}$, we know that
$$[\operatorname{Der}_m(kP), \operatorname{Der}_n(kP)] \subset \Der_{m+n-1}(kP),$$
so then by induction one can easily prove that $D^n(\Der_2(kP)) \subset \Der_{2^{n}+1}(kP)$. Of course, $\nolinebreak\Der_{\ell\ell(kP)}(kP)=0$, so we get our bound.
\end{proof}
\begin{proposition} If $l(\Gamma(P))<n+1$ then $\HH^1(kP)$ is $n$-strongly solvable. If $l(\Gamma(P)) > n+1$ then $\HH^1(P)$ is not $n$-strongly solvable.
\end{proposition}
\begin{proof} Again for ease we work with $\bar\Gamma(P)$. If $l(\bar\Gamma(P))<n$ then we see that $D^n(\mathfrak{h})$ is zero. So by Proposition \ref{nstrongsolv} we have that $\HH^1(kP)$ is $n$-strongly solvable. \\

If $l(\bar\Gamma(P))>n$, then we must have some $b \in T_{n+1}$, and so some $a \in T_n$ with an edge into $b$. But clearly, from our last proposition we see that
$$\operatorname{Ad}(\mathfrak{h}_a)^p(\mathfrak{h}_b) = \mathfrak{h}_b,$$
so that $\mathfrak{h}_b \subset C^m(D^n(\mathfrak{h}))$ for all $m$. So $\mathfrak{h}$ is not $n$-strongly solvable, and so $\HH^1(kP)$ is not $n$-strongly solvable.
\end{proof}
Write $\operatorname{ss-rank}(\mathfrak{g}) := \operatorname{min}\{n: \mathfrak{g} \textrm{ is } n \textrm{ strongly solvable} \}$, for the strongly-solvable rank of $\mathfrak{g}$. 
\begin{corollary} 
$$\operatorname{ss-rank}(\HH^1(kP)) \in \{l(\Gamma(P)), l(\Gamma(P))-1\}.$$
\end{corollary}
\section{Examples}
Having done the hard work, we can now enjoy some pretty pictures.
\begin{example}[Elementary abelian] When $P$ is elementary abelian, $\mathfrak{h}_a := \mathfrak{g}_a$, and for each pair $a$, $b$, there is $\phi \in \mathfrak{h}_a$ so that $\phi(b)\not=0$, unless $b=e$. So $\Gamma(P)$ contains all edges other than those into $e$. When $P$ is abelian, we have exactly the same graph on $P / \Phi(P)$.
\end{example}

\begin{example}[$\operatorname{UT}(3,3)$] The smallest non-abelian examples are the two non-abelian groups of size 27. Choosing $P:=\operatorname{UT}(3,3)$ we get the following graph:

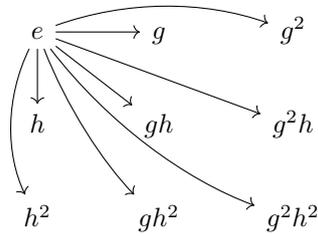
\[\begin{tikzcd}
	{ e} & { g} & {g^2} \\
	h & gh & {g^2h} \\
	{h^2} & {gh^2} & {g^2h^2}
	\arrow[from=1-1, to=1-2]
	\arrow[curve={height=-12pt}, from=1-1, to=1-3]
	\arrow[from=1-1, to=2-1]
	\arrow[from=1-1, to=2-2]
	\arrow[from=1-1, to=2-3]
	\arrow[curve={height=12pt}, from=1-1, to=3-1]
	\arrow[curve={height=6pt}, from=1-1, to=3-2]
	\arrow[curve={height=12pt}, from=1-1, to=3-3]
\end{tikzcd}\]
Which obviously has no cycles, and only paths of length 1. 
\end{example}
\begin{example}[$M27$] Now, choosing $P=M27$, the other smallest non-abelian example we get a slightly more complex graph. 
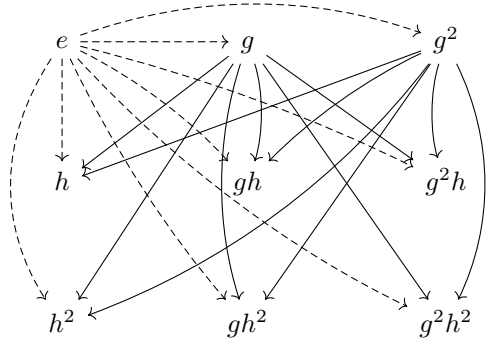
\[\begin{tikzcd}
	e && g && {g^2} \\
	\\
	h && gh && {g^2h} \\
	\\
	{h^2} && {gh^2} && {g^2h^2}
	\arrow[dashed, from=1-1, to=1-3]
	\arrow[curve={height=-18pt}, dashed, from=1-1, to=1-5]
	\arrow[dashed, from=1-1, to=3-1]
	\arrow[curve={height=-6pt}, dashed, from=1-1, to=3-3]
	\arrow[curve={height=-6pt}, dashed, from=1-1, to=3-5]
	\arrow[curve={height=24pt}, dashed, from=1-1, to=5-1]
	\arrow[curve={height=6pt}, dashed, from=1-1, to=5-3]
	\arrow[curve={height=12pt}, dashed, from=1-1, to=5-5]
	\arrow[from=1-3, to=3-1]
	\arrow[curve={height=-6pt}, from=1-3, to=3-3]
	\arrow[from=1-3, to=3-5]
	\arrow[from=1-3, to=5-1]
	\arrow[curve={height=12pt}, from=1-3, to=5-3]
	\arrow[from=1-3, to=5-5]
	\arrow[from=1-5, to=3-1]
	\arrow[curve={height=6pt}, from=1-5, to=3-3]
	\arrow[curve={height=6pt}, from=1-5, to=3-5]
	\arrow[curve={height=-18pt}, from=1-5, to=5-1]
	\arrow[from=1-5, to=5-3]
	\arrow[curve={height=-18pt}, from=1-5, to=5-5]
\end{tikzcd}\]
As $e$ is minimal, we have drawn the paths from $e$ as dotted to reduce clutter. As $\Gamma(P)$ has no cycles, $\HH^1(kP)$ is solvable.
\end{example}
\begin{example}[$\mathbb{Z}/9\mathbb{Z} \rtimes \mathbb{Z}/9\mathbb{Z}$] By this group we mean exactly the group with Small group identifier $(81,4)$. The graph $\Gamma(P)$ has $56$ edges. It can be seen simply by this number that there must be a cycle; however looking only at the top corner, we see the following,
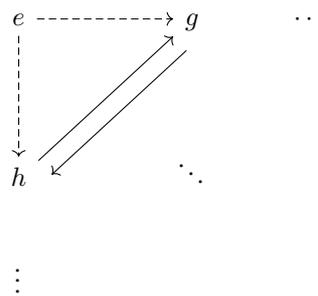
\[\begin{tikzcd}
	e && g & \cdots \\
	\\
	h && \ddots \\
	\vdots
	\arrow[dashed, from=1-1, to=1-3]
	\arrow[dashed, from=1-1, to=3-1]
	\arrow[shift left=3, from=1-3, to=3-1]
	\arrow[from=3-1, to=1-3]
\end{tikzcd}\]
so that $\HH^1(kP)$ cannot be solvable.
\end{example}
\begin{example} We of course are also interested in non-$p$-groups, where our theorem gives us a sufficient criterion for insolvability. Considering groups of small orders, the first group that can be shown to be insolvable is $\operatorname{SL}(2,3)$ at the prime $3$. Considering $\operatorname{SL}(2,3)/\Phi_3(\operatorname{SL}(2,3))$, we get the following graph,
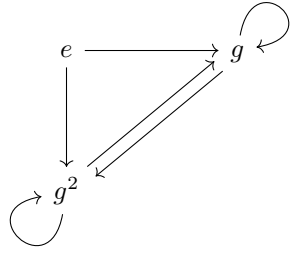
\[\begin{tikzcd}
	e && g \\
	\\
	{g^2}
	\arrow[from=1-1, to=1-3]
	\arrow[from=1-1, to=3-1]
	\arrow[from=1-3, to=1-3, loop, in=10, out=80, distance=10mm]
	\arrow[shift left, from=1-3, to=3-1]
	\arrow[shift left, from=3-1, to=1-3]
	\arrow[from=3-1, to=3-1, loop, in=190, out=260, distance=10mm]
\end{tikzcd}\]
which obviously contains a cycle.
\end{example}
\section{Solvability of Product Groups}
\subsection{Transfer on Product Groups}
We now use the main theorem to prove the following corollary.
\begin{corollary} Given $p\not=2$ and $p$-groups $P_1, P_2$, and writing $P:=P_1 \times P_2$, $\HH^1(k P)$ is solvable if and only if $\HH^1(kP_1)$ and $\HH^1(kP_2)$ are solvable. \end{corollary}
First note that one of the directions is trivial. $\operatorname{HH}^1(kP_i) \subset \HH^1(kP)$, so that solvability of $\operatorname{HH}^1(kP)$ implies solvability of $\HH^1(kP_i)$ for $i\in \{1, 2\}$. Thus we only need to show the converse; so we suppose that both $\HH^1(kP_1)$ and $\HH^1(kP_2)$ are solvable. At this point, we recall a result from $\cite{BKL}$: if $p\not=2$ and $Z(P) \not \subset \Phi(P)$ then $\HH^1(kP)$ is not solvable. So then we must have $Z(P_1) \subset \Phi(P_1)$ and $Z(P_2) \subset \Phi(P_2)$. We proceed with the following proposition.
\begin{proposition} Given $x \in P_1$, $y \in P_2$, we have
$$\operatorname{tr}_{C_P((x,y))}^P(a,b) = (\operatorname{tr}_{C_{P_1}(x)}^{P_1}(a)^{|y^{P_2}|}, \operatorname{tr}_{C_{P_2}(y)}^{P_2}(b)^{|x^{P_1}|}).$$
\end{proposition}
\begin{proof}
First, we note that 
$$C_P((x,y)= C_{P_1}(x) \times C_{P_2}(y).$$
The inclusion $C_{P_1}(x) \times C_{P_2}(y)$ into $C_P((x,y))$ is immediate. Note that if $(a,b)$ commutes with $(x,y)$ then taking the projection onto  $P_1$, $a$ will necessarily commute with $x$. The same obviously holds for $b$ and $y$, so that we get the reverse inclusion. \\
Take $\gamma_1, \ldots, \gamma_n$ a right transversal of $C_{P_1}(x)$ in $P_1$, and $\delta_1, \ldots, \delta_m$ a right transversal of $C_{P_2}(y)$ in $P_2$. Then clearly $\{( \gamma_i ,\delta_j ): i \in \{1,\ldots,n\}, j \in \{1, \ldots, m \}\}$ is a right transversal of $C_P((x,y))$. So given $(a,b)\in P$, and $1\leq i \leq n$, $1 \leq j \leq m$, to define transfer we require the elements $(a_{[ij]}, b_{[ij]}) \in C_P((x,y))$ so that
$$(\gamma_i, \delta_j) (a,b) = (a_{[ij]}, b_{[ij]}) (\gamma_{s_{ij}}, \delta_{t_{ij}}),$$
where then $\operatorname{tr}_{C_P((x,y))}^P(a,b) = \prod (a_{[ij]}, b_{[ij]})$.
But we note, given $a_{[i]}\in C_{P_1}(x)$ and $b_{[i]} \in C_{P_2}(y)$, with
$$\begin{aligned}\gamma_i a &= a_{[i]} \gamma_{s_i}, \\
\delta_j b &= b_{[j]} \delta_{t_j}, \end{aligned}$$
 we see our requirement on $a_{[ij]}, b_{[ij]}$ is $(\gamma_i a, \delta_j b) = (a_{[ij]}\gamma_{s_{ij}}, b_{[ij]}\delta_{t_{ij}}),$ which coordinate-wise is satisfied by $a_{[ij]}=a_{[i]}$, $b_{[ij]}=b_{[j]}$. So we get our result.
 \end{proof}
Then we immediately get the following fact.
\begin{corollary} Suppose $(x,y) \in P$, and that neither $x$ nor $y$ are central. Then $\operatorname{tr}_{C_P((x,y))}^P$ is zero. \end{corollary}
\subsection{The graph on Product Groups}
We now consider the structure of the graph $\Gamma(P)$. $\Gamma(P)$ has vertices $P/\Phi(P)$; we know that $\Phi(P) = \Phi(P_1) \times \Phi(P_2)$. So then $P / \Phi(P) \cong P_1 / \Phi(P_1) \times P_2 / \Phi(P_2)$. \\
\begin{proposition} Given $P_1, P_2, P$ so that $\HH^1(kP_i)$ is solvable, $\bar{\Gamma}(P)$ only has edges from the inclusions of subgraphs
$$\bar{\Gamma}(P_1) \hookrightarrow \bar{\Gamma}(P), \; \bar{\Gamma}(P_2) \hookrightarrow \bar{\Gamma}(P),$$
where the inclusions on vertices are given by the inclusion of subgroups.
\end{proposition}
\begin{proof}
To be thorough, we split this proof into $4$ parts. First, we show that our inclusion of subgraphs is well defined. Second, we show that these subgraph inclusions are ``full'' (there are no extra edges in these subgraphs). Third, we show that there are no edges outside of these subgraphs, and finally we show that there are no edges between these two subgraphs.
\begin{enumerate}
\item We begin by showing that these subgraphs are well defined. Suppose I have an edge from $a$ to $b$ in $\bar{\Gamma}(P_1)$. This is the data of $x,y \in P_1$ so that $x$ (resp. $y$) has image $a$ (resp. $b$) in $P_1 / \Phi(P_1)$, and a map $\alpha:C_{P_1}(x) \rightarrow k$ so that
$$\alpha \circ \operatorname{tr}_{C_{P_1}(x)}^{P_1}(y) \not=0.$$
So then, we seek to demonstrate an edge between $(a,e)$ and $(b,e)$. Take $(x,e), (y,e) \in P$, and $(\alpha,0): C_{P_1}(x)\times P_2 \cong C_P((x,e)) \rightarrow k$. Then we can see from our formula for transfer that 
$$(\alpha,0)\circ \operatorname{tr}_{C_P((x,e))}^P(y,e) = (\alpha,0)(\operatorname{tr}_{C_{P_1}(x)}^{P_1}(y), e) = \alpha \circ \operatorname{tr}_{C_{P_1}(x)}^P(y),$$
which doesn't vanish. So we see that $\bar{\Gamma}(P_1), \bar{\Gamma}(P_2)$ are subgraphs of $\bar{\Gamma}(P)$. \\
\item We seek to show that no extra edges are added to these subgraphs. First suppose we have an edge from $(a,e)$ to $(b,e)$ in $\bar{\Gamma}(P)$. This is the data of $(x,y), (w,z) \in P$ so that $(x,y)$ (resp. $(w,z)$) has image $(a,e)$ (resp. $(b,e)$), and of $\alpha:C_P{((x,y))}\rightarrow k$, so that with $\phi:= \alpha \circ \operatorname{tr}_{C_P((x,y))}^P$, we have
$$\phi((w,z)) \not =0.$$
But note, this is $\phi(w,e)+\phi(e,z)$. Further, $z\in \Phi(P_2)$ so that $\phi(e,z) =0$. So we know that $\phi(w,e) \not=0$. So then writing $\beta:C_{P_1}(x) \rightarrow k$ for the composition of $\alpha$ with the inclusion $C_{P_1}(x) \rightarrow C_P((x,y))$ we have that $\beta$ realises an edge between $a$ and $b$ in $\bar{\Gamma}(P)$.

\item Take $(a,b) \in P/ \Phi(P)$, and suppose that $a\not=e, b\not=e$, and that $(x,y) \in P$ is a preimage of $(a,b)$ in $P$. We know that $Z(P_1) \subset \Phi(P_1)$ and $Z(P_2) \subset \Phi(P_2)$, so we know that neither $x$ nor $y$ is central in $P$. So then $\operatorname{tr}_{C_P((x,y))}^P$ is zero, and thus $\mathfrak{h}_{(a,b)}$ is zero. So in the reduced graph $\bar{\Gamma}{(P)}$, there are no edges into or out of $(a,b)$. \\
\item Now take $a\in P_1$, $b\in P_2$, both not the identity. We seek to show that we have no edges from an element $(a,e)$ to one of the form $(e,b)$. Such an edge would imply the existence $(x,y), (w,z)\in P$ and $\alpha:C_{P}(a,e) \rightarrow k$ with the following properties. We have that $(x,y)$ (resp. $(w,z)$) has image $(a,e)$ (resp. $(e,b)$) in $P/\Phi(P)$, so that with $\phi := \alpha \circ \operatorname{tr}_{C_P((x,y))}$, we have
$$\phi((w,e)) + \phi((e,z)) = \phi((w,z))\not=0.$$
Now, from before we know that
$$\operatorname{tr}_{C_P((x,y))}^P((w,z)) = (\operatorname{tr}_{C_{P_1}(x)}^{P_1}(w)^{|y^{P_2}|}, \operatorname{tr}_{C_{P_2}(y)}^{P_1}(z)^{|x^{P_1}|}),$$
But as $x$ is not central, clearly the second coordinate is a $p^{th}$ power so the term $\phi((e,z))$ will vanish. But further, we know that $w \in \Phi(P_1)$, so that $\phi((w,e))$ will also vanish. So we cannot have such an edge. \\
\end{enumerate}
\end{proof}

So then we see that $\bar{\Gamma}(P)$ has a cycle if and only if one of $\bar{\Gamma}(P_1)$ and $\bar{\Gamma}(P_2)$ has one.

\section{Partial results on $2$-groups}
As stated before, our proof does not work in the case $p=2$, as we cannot rule out that the graph has loops. To state a partial result, we will define a second graph, $\Gamma_2(P)$; if this has graph has any cycles then our $\HH^1(kP)$ will not be solvable. For the rest of the section, we assume $p=2$. We define the graph as follows.
\begin{defn} $\Gamma_2(P)$ is a directed graph with:
\begin{itemize} 
\item vertices labelled by $P / \Phi(P)$.
\item an edge $a \rightsquigarrow b$, when there are edges $ a\rightarrow b$ and $a\rightarrow ab$ in $\Gamma(P)$
\end{itemize}
\end{defn}
We have the following relatively easy calculation
\begin{lemma} If we have $a, b \in P/ \Phi(P)$ and an edge $a \rightsquigarrow b$ in $\Gamma_2(P)$, then writing $\mathfrak{a}$ for the subalgebra generated by $\mathfrak{h}_a, \mathfrak{h}_b$, we have that $\mathfrak{h}_b \subset [\mathfrak{a}, \mathfrak{a}]$. \end{lemma}
\begin{proof} Suppose we are in the situation above. Then we have edges $a \rightarrow b$ and $a \rightarrow ab$ in $\Gamma(P)$, which give us elements $(\phi,a), (\phi',a) \in \mathfrak{h}_a$ so that $\phi(b)\not=0, \phi'(ab)\not=0$. We seek $(\theta, a)$ so that $\theta(b)\not =0$ and $\theta(ab)\not=0$. If neither of $\phi, \phi'$ satisfy this, then
$$\begin{aligned}(\phi+\phi')(b)&=\phi(b)\not=0, \\
(\phi+\phi')(ab)&=\phi'(ab)\not=0 .\end{aligned}$$
So we can assume without restriction of generality that $\phi(a)\not=0, \phi(ab)\not=0$. Take any $(\psi,b) \in \mathfrak{h}_b$. Then we have:
$$[(\phi,a),(\psi,b)] = (\phi(b)\psi-\psi(a)\phi, ab),$$
and further
$$\begin{aligned} \operatorname{Ad}^2((\phi,a))((\psi,b))&= [(\phi,a), (\phi(b)\psi-\psi(a)\phi,ab)] \\
&= [(\phi(ab)(\phi(b)\psi-\psi(a)\phi)-(\phi(b)\psi(a)-\psi(a)\phi(a))\phi, b] \\
&= (\phi(ab)\phi(b)\psi, b). \end{aligned} $$
So as $\phi(ab)\not=0$ and $\phi(b)\not=0$, we have that $(\psi, b) \subset [\mathfrak{a}, \mathfrak{a}]$. But as $(\psi,b)$ was arbitrary, we have $\mathfrak{h}_a \subset [ \mathfrak{a}, \mathfrak{a}]$. 
\end{proof}
\begin{corollary} If $\Gamma_2(P)$ has a cycle, then $\HH^1(kP)$ is not solvable.
\end{corollary}
\begin{proof}
Suppose we have a cycle $a_1 \rightsquigarrow a_2 \rightsquigarrow \ldots \rightsquigarrow a_n \rightsquigarrow a_1$. Then write $\mathfrak{a}$ for the subalgebra generated by $\{\mathfrak{h}_{a_i}\}_{i=1,\ldots,n}$. From the above, we clearly have that $\mathfrak{h}_{a_i} \subset [\mathfrak{a}, \mathfrak{a}]$ for each $i=1,\ldots,n$, but as these generate $\mathfrak{a}$, $[\mathfrak{a}, \mathfrak{a}] = \mathfrak{a}$. So $\mathfrak{h}$ is not solvable, and so neither is $\HH^1(kP)$.
\end{proof}
\begin{remark} It is not obvious that the converse of this statement holds. Indeed, all examples the author has checked only fail to be solvable in the case where $\Gamma_2(P)$ has a cycle.
\end{remark}
\begin{example}[$C_2$] We briefly make explicit the case of $C_2$, which exemplifies that our old method of proof does not work. $\Gamma(C_2)$ is as follows:
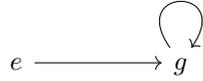
\[\begin{tikzcd}
	e && g
	\arrow[from=1-1, to=1-3]
	\arrow[from=1-3, to=1-3, loop, in=55, out=125, distance=10mm]
\end{tikzcd}\]
and of course, $\Gamma_2(C_2)$ has no edges.
\end{example}
\begin{example}[$C_2 \times C_2$] This is our first example of a non-solvable Lie algebra in the case $p=2$. Below is the graph $\Gamma(C_2 \times C_2)$ 

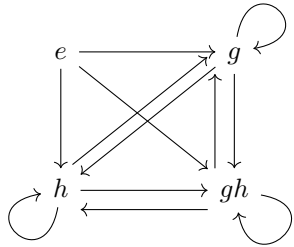
\[\begin{tikzcd}
	e && g \\
	\\
	h && gh
	\arrow[from=1-1, to=1-3]
	\arrow[from=1-1, to=3-1]
	\arrow[from=1-1, to=3-3]
	\arrow[from=1-3, to=1-3, loop, in=10, out=80, distance=10mm]
	\arrow[from=1-3, to=3-1]
	\arrow[from=1-3, to=3-3]
	\arrow[shift left=2, from=3-1, to=1-3]
	\arrow[from=3-1, to=3-1, loop, in=190, out=260, distance=10mm]
	\arrow[from=3-1, to=3-3]
	\arrow[shift left=3, from=3-3, to=1-3]
	\arrow[shift left=3, from=3-3, to=3-1]
	\arrow[from=3-3, to=3-3, loop, in=280, out=350, distance=10mm]
\end{tikzcd}\]

which gives us the graph $\Gamma_2(C_2 \times C_2)$ as follows:

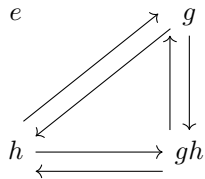
\[\begin{tikzcd}
	e && g \\
	\\
	h && gh
	\arrow[from=1-3, to=3-1]
	\arrow[from=1-3, to=3-3]
	\arrow[shift left=3, from=3-1, to=1-3]
	\arrow[from=3-1, to=3-3]
	\arrow[shift left=3, from=3-3, to=1-3]
	\arrow[shift left=3, from=3-3, to=3-1]
\end{tikzcd}\]
which obviously contains cycles, so that $\operatorname{HH}^1(k(C_2 \times C_2))$ cannot be solvable.
\end{example}
\begin{example}[The Small Group $(32,27)$] We give one more example with a larger graph $\Gamma(P)$, to drive home the point. $\Gamma(P)$ is as follows:
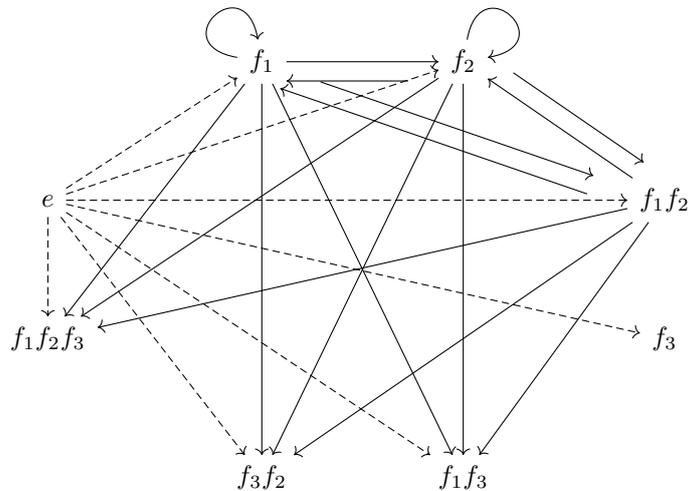
\[\begin{tikzcd}
	&& {f_1} && {f_2} \\
	\\
	e &&&&&& {f_1f_2} \\
	\\
	{f_1f_2f_3} &&&&&& {f_3} \\
	\\
	&& {f_3f_2} && {f_1f_3}
	\arrow[from=1-3, to=1-3, loop, in=100, out=170, distance=10mm]
	\arrow[from=1-3, to=1-5]
	\arrow[between={0.1}{0.9}, from=1-3, to=3-7]
	\arrow[from=1-3, to=5-1]
	\arrow[from=1-3, to=7-3]
	\arrow[from=1-3, to=7-5]
	\arrow[shift left=3, between={0.2}{1}, from=1-5, to=1-3]
	\arrow[from=1-5, to=1-5, loop, in=10, out=80, distance=10mm]
	\arrow[shift left=3, between={0.1}{1}, from=1-5, to=3-7]
	\arrow[from=1-5, to=5-1]
	\arrow[from=1-5, to=7-3]
	\arrow[from=1-5, to=7-5]
	\arrow[dashed, from=3-1, to=1-3]
	\arrow[dashed, from=3-1, to=1-5]
	\arrow[dashed, from=3-1, to=3-7]
	\arrow[dashed, from=3-1, to=5-1]
	\arrow[dashed, from=3-1, to=5-7]
	\arrow[dashed, from=3-1, to=7-3]
	\arrow[dashed, from=3-1, to=7-5]
	\arrow[shift left=3, between={0.1}{1}, from=3-7, to=1-3]
	\arrow[from=3-7, to=1-5]
	\arrow[from=3-7, to=5-1]
	\arrow[from=3-7, to=7-3]
	\arrow[from=3-7, to=7-5]
\end{tikzcd}\]
then, not drawing vertices with no outgoing edges, this leads to $\Gamma_2(P)$ to be of the form:

\[\begin{tikzcd}
	{f_1} && {f_2} \\
	\\
	&&&& {f_1f_2}
	\arrow[from=1-1, to=1-3]
	\arrow[between={0.1}{0.9}, from=1-1, to=3-5]
	\arrow[shift left=3, between={0.2}{1}, from=1-3, to=1-1]
	\arrow[shift left=3, between={0.1}{1}, from=1-3, to=3-5]
	\arrow[shift left=3, between={0.1}{1}, from=3-5, to=1-1]
	\arrow[from=3-5, to=1-3]
\end{tikzcd}\].
\end{example}
\bibliographystyle{alpha} 
\bibliography{referencesSOLP} 

@book{Witherspoon,
	author = {Witherspoon, Sarah},
	doi = {10.1090/gsm/204},
	journal = {Graduate Studies in Mathematics},
	month = dec,
	publisher = {American Mathematical Society},
	title = {Hochschild Cohomology for Algebras},
	url = {http://dx.doi.org/10.1090/gsm/204},
	year = {2019}}

@book{Gruenberg,
	author = {Gruenberg, Karl W.},
	doi = {10.1007/bfb0059162},
	journal = {Lecture Notes in Mathematics},
	language = {en},
	publisher = {Springer Berlin Heidelberg},
	title = {Cohomological Topics in Group Theory},
	url = {http://dx.doi.org/10.1007/BFb0059162},
	year = {1970}}

@book{Huppert,
	author = {Huppert, Bertram},
	doi = {10.1007/978-3-031-87529-8},
	journal = {Grundlehren der mathematischen Wissenschaften},
	language = {en},
	publisher = {Springer Nature Switzerland},
	title = {Finite Groups I},
	url = {http://dx.doi.org/10.1007/978-3-031-87529-8},
	year = {2025}}

@article{Eisele,
	abstractnote = {<p>For an arbitrary finite-dimensional algebra <inline-formula content-type="math/mathml"> <mml:math xmlns:mml="http://www.w3.org/1998/Math/MathML" alttext="upper A"> <mml:semantics> <mml:mi>A</mml:mi> <mml:annotation encoding="application/x-tex">A</mml:annotation> </mml:semantics> </mml:math> </inline-formula>, we introduce a general approach to determining when its first Hochschild cohomology <inline-formula content-type="math/mathml"> <mml:math xmlns:mml="http://www.w3.org/1998/Math/MathML" alttext="normal upper H normal upper H Superscript 1 Baseline left-parenthesis upper A right-parenthesis"> <mml:semantics> <mml:mrow> <mml:msup> <mml:mrow class="MJX-TeXAtom-ORD"> <mml:mi mathvariant="normal">H</mml:mi> <mml:mi mathvariant="normal">H</mml:mi> </mml:mrow> <mml:mn>1</mml:mn> </mml:msup> <mml:mo stretchy="false">(</mml:mo> <mml:mi>A</mml:mi> <mml:mo stretchy="false">)</mml:mo> </mml:mrow> <mml:annotation encoding="application/x-tex">mathrm {HH}^1(A)</mml:annotation> </mml:semantics> </mml:math> </inline-formula>, considered as a Lie algebra, is solvable. If <inline-formula content-type="math/mathml"> <mml:math xmlns:mml="http://www.w3.org/1998/Math/MathML" alttext="upper A"> <mml:semantics> <mml:mi>A</mml:mi> <mml:annotation encoding="application/x-tex">A</mml:annotation> </mml:semantics> </mml:math> </inline-formula> is, moreover, of tame or finite representation type, we are able to describe <inline-formula content-type="math/mathml"> <mml:math xmlns:mml="http://www.w3.org/1998/Math/MathML" alttext="normal upper H normal upper H Superscript 1 Baseline left-parenthesis upper A right-parenthesis"> <mml:semantics> <mml:mrow> <mml:msup> <mml:mrow class="MJX-TeXAtom-ORD"> <mml:mi mathvariant="normal">H</mml:mi> <mml:mi mathvariant="normal">H</mml:mi> </mml:mrow> <mml:mn>1</mml:mn> </mml:msup> <mml:mo stretchy="false">(</mml:mo> <mml:mi>A</mml:mi> <mml:mo stretchy="false">)</mml:mo> </mml:mrow> <mml:annotation encoding="application/x-tex">mathrm {HH}^1(A)</mml:annotation> </mml:semantics> </mml:math> </inline-formula> as the direct sum of a solvable Lie algebra and a sum of copies of <inline-formula content-type="math/mathml"> <mml:math xmlns:mml="http://www.w3.org/1998/Math/MathML" alttext="German s German l Subscript 2"> <mml:semantics> <mml:msub> <mml:mrow class="MJX-TeXAtom-ORD"> <mml:mi mathvariant="fraktur">s</mml:mi> <mml:mi mathvariant="fraktur">l</mml:mi> </mml:mrow> <mml:mn>2</mml:mn> </mml:msub> <mml:annotation encoding="application/x-tex">mathfrak {sl}_2</mml:annotation> </mml:semantics> </mml:math> </inline-formula>. We proceed to determine the exact number of such copies, and give an explicit formula for this number in terms of certain chains of Kronecker subquivers of the quiver of <inline-formula content-type="math/mathml"> <mml:math xmlns:mml="http://www.w3.org/1998/Math/MathML" alttext="upper A"> <mml:semantics> <mml:mi>A</mml:mi> <mml:annotation encoding="application/x-tex">A</mml:annotation> </mml:semantics> </mml:math> </inline-formula>. As a corollary, we obtain a precise answer to a question posed by Chaparro, Schroll, and Solotar.</p>},
	author = {Eisele, Florian and Raedschelders, Theo},
	doi = {10.1090/tran/8064},
	journal = {Transactions of the American Mathematical Society},
	language = {en},
	month = sept,
	number = {11},
	pages = {7607--7638},
	publisher = {American Mathematical Society (AMS)},
	title = {On solvability of the first Hochschild cohomology of a finite-dimensional algebra},
	url = {http://dx.doi.org/10.1090/TRAN/8064},
	volume = {373},
	year = {2020}}

@article{Liu,
	abstractnote = {<jats:p>We realize explicitly the well-known additive decomposition of the Hochschild cohomology ring of a group algebra at the chain level. As a result, we describe the cup product, the Batalin--Vilkovisky operator and the Lie bracket in the Hochschild cohomology ring of a group algebra.</jats:p>},
	author = {Liu, Yuming and Zhou, Guodong},
	doi = {10.4171/jncg/249},
	journal = {Journal of Noncommutative Geometry},
	month = sept,
	number = {3},
	pages = {811--858},
	publisher = {European Mathematical Society - EMS - Publishing House GmbH},
	title = {The Batalin--Vilkovisky structure over the Hochschild cohomology ring of a group algebra},
	url = {http://dx.doi.org/10.4171/JNCG/249},
	volume = {10},
	year = {2016}}

@article{Degrassi,
	author = {Lleonard Rubio y Degrassi and Sibylle Schroll and Andrea Solotar},
	journal = {Quaestiones Mathematicae},
	pages = {1955 - 1980},
	title = {The first Hochschild cohomology as a Lie algebra},
	url = {https://api.semanticscholar.org/CorpusID:85542677},
	volume = {46},
	year = {2019}}

@article{BKL,
	author = {Benson, Dave and Kessar, Radha and Linckelmann, Markus},
	doi = {10.2140/pjm.2021.313.1},
	issn = {0030-8730},
	journal = {Pacific Journal of Mathematics},
	month = sep,
	number = {1},
	pages = {1--44},
	publisher = {Mathematical Sciences Publishers},
	title = {On the BV structure of the Hochschild cohomology of finite group algebras},
	url = {http://dx.doi.org/10.2140/pjm.2021.313.1},
	volume = {313},
	year = {2021}}

@misc{Wang,
	archiveprefix = {arXiv},
	author = {Markus Linckelmann and Jialin Wang},
	eprint = {2503.16697},
	howpublished = {(Available at https://arxiv.org/abs/2503.16697)},
	primaryclass = {math.RT},
	title = {On the solvability of the Lie algebra $\mathrm{HH}^1(B)$ for blocks of finite groups},
	url = {https://arxiv.org/abs/2503.16697},
	year = {2025}}
\end{document}